\pdfoutput=1
\documentclass[graybox]{svmult}
\usepackage{url}
\usepackage{amsfonts,amsmath,comment}
\usepackage{eurosym}
\usepackage{booktabs}
\usepackage{caption} 

\usepackage{float}
\usepackage{makeidx}         
\usepackage{graphicx}        
\usepackage{subcaption}
\usepackage{multicol}        

\makeindex             

\begin{document}

\title*{The petrol station game: the regional average price. A mean field analysis}
\author{Fabio Bagagiolo
\and 
Federico Pontiroli
\and
Ivan Romanò
}
\institute{Fabio Bagagiolo \at Dept. of Mathematics, University of Trento, Italy; \email{fabio.bagagiolo@unitn.it}
\and Federico Pontiroli \at Dept. of Mathematics, University of Trento, Italy; \email{federico.pontiroli@studenti.unitn.it}
\and Ivan Romanò \at Dept. of Mathematics, University of Trento, Italy; \email{ivan.romano@unitn.it}}
%
%
\maketitle

\abstract*{In January 2023, the obligation for petrol stations to display the average fuel price calculated on a regional basis was introduced by the Italian Government. A mean field game model is here proposed to describe the corresponding possible evolution of the fuel price. In particular, each petrol station decides its own price, for the next day, taking into account of an estimate of the average price (the mean field) which will be communicated by the Local Authority. The existence and the evolution of a mean field equilibrium are analytically investigated. Some simulations are compared with real data from the Autonomous Province of Trento, Italy.}

\abstract{In January 2023, the obligation for petrol stations to display the average fuel price calculated on a regional basis was introduced by the Italian Government. A mean field game model is here proposed to describe the evolution of the fuel price. In particular, each petrol station decides its own price, for the next day, taking account of an estimate of the average price, to be communicated by the Local Authority. The existence and the evolution of a mean field equilibrium are analytically investigated. Some simulations are compared with real data from the Autonomous Province of Trento, Italy.}


\section{Introduction}
\label{sec:1}
In August 2023, it became mandatory for petrol stations in Italy to publicly display the average regional fuel prices. Such obligation was then suspended during 2024. However, data collection and processing continued, available on \cite{osservaprezzi}. This was due to an Italian law designed, principally, to increase
transparency and fuel  price control. The underlying concept
is: the ``rational" consumer will avoid filling up at outlets
whose prices are significantly above the average, knowing that for every gas station that has
above-average prices, there is at least another one that has cheaper prices.
In practice, the issue is more complex, because other factors may influence the choice of the customers, such as the presence of a bar, a car wash, a repair shop, etc..
The average price to be displayed is updated daily based on the individual
price that each gas station has communicated to the Ministry of Enterprise and Made in Italy (MIMIT)
the day before. It is then a quantity with a degree of uncertainty, until it is published, and the station attendant
must estimate it to understand what price is best for them.

In this paper we interpret the problem as a mean-field game, in which the mean field is the average price. In particular, in the formulation of the problem, we get inspiration from a simple model that regulates the start time
of a company meeting, ``When does the meeting start", which is contained in \cite{liolasgue} (see \cite{laslio} for the seminal paper on mean field games). It is assumed that the
gas station attendants, and consequently the consumers, are rational decision makers: the players' decisions are assumed to minimize some cost function.
In our main model we assume that the cost of the raw material is constant and is the same for
everyone. This may be justified since the evolution  is considered in a rather short window of days (see also Remark \ref{rmrk:minimum}). Moreover, our model tries to simulate transparent and ethical behaviour rather than the formation of prices in a complex market.


The model then aims to describe the dynamics of decisions, that is
the evolution of prices, and  therefore to rationally interpret the process that
leads each gas station attendant to choose a certain price, in order to be capable to possibly predict future behaviors, suitably fitting the parameters. We refer to \cite{dal} for a more statistical viewpoint on the same problem.


This Introduction continues with some further details on the legislative context, the available data and a first notational description of the price evolution. Then, in Section \ref{sec:analytic}, an analytical model is constructed and studied for what concerns equilibria and their convergence. In Section \ref{sec:simulations} some simulations are presented and commented, comparing with data from the Autonomous Province of Trento (PAT) and after a possible clustering of the petrol stations based on their daily price on July $1^{st}$.
As general references for the numerical and statistical tools, we refer to \cite{nowrig}, \cite{hastibfri}.

{\it \underline{Law \& Model}.}
With the Legislative Decree of January 14, 2023 (see \cite{decretolegge} for the full text), the government
communicated the new ``Provisions regarding fuel bonuses and transparency
and control of the retail price of automotive fuel." 
Here we report the most significant aspects as in the FAQ section
(see \cite{faq}) of the MIMIT website.

{\it MIMIT, upon receiving the price
communications, processes the data and calculates the arithmetic mean, on a regional
and autonomous province basis, of the prices communicated by those engaged in the
sale of automotive fuel to the public at facilities
located outside the motorway network [...]}
{\it MIMIT, starting August $1^{st}$, 2023, will publish average prices in open format daily by 8:30 a.m. in a dedicated section of its institutional website, at www.mimit.gov.it/it/prezzomedio-
carburanti [...].}
{\it Gas station attendants
must display a sign showing the average prices
for the types of fuel available at their point of sale,
ensuring that it is updated daily.}

Let us suppose that we have fixed a data collection day. We are going to use the following notation: $\overline p$ as the average price displayed on that day (and communicated by the Ministry early in the morning); $p_i^+$ as the price decided by the $i$-th player (the $i$-th petrol station) for the next day for her/his own station; $\overline p^+$ as the average price of the next day calculated as arithmetic mean of all prices $p_i^+$. The transition scheme is then

\begin{equation}
\label{eq:transition}
\overline p\ \longrightarrow\ \left(p_i^+\right)_i\ \longrightarrow\ \overline p^+
\end{equation}

From the game theory point of view (see for example \cite{pet}), we have a non-cooperative game
with imperfect information (no one knows the prices that
others will set, and consequently the average price, until the following day) and complete information
(ignorance of the others' moves is the same for each player; the average price of the
current day is known to all).

The data has been collected and made available as a .csv file since 2015 on the website \cite{dati}.
Actually, starting in 2023, MIMIT has performed a data processing function and
made a summary available. 


\section{The analytical model}
\label{sec:analytic}

In this section we introduce the costs faced by the agent (player) $i$-th when the price of the next day $p_i^+$ must be decided, based on the estimated average price $\overline p^+$, to describe how the transition (\ref{eq:transition}) may be generated.

What is uncertain in the decision process of the $i$-th player is the knowledge of the average price of the next day, $\overline p^+$. We then assume that the single agent $i$-th, in its decision process,  uses a personal estimate of $\overline p^+$ as

\begin{equation}
\label{eq:tildap}
\tilde p_i:=\overline p+\sigma_i\varepsilon_i
\end{equation}

\noindent
where $\sigma_i>0$ is a fixed coefficient and $\varepsilon^i\sim N(0,1)$ is a normally distributed random variable with $\varepsilon^i$ and $\varepsilon^j$ independent when $i\neq j$. We now introduce the costs faced by the agent $i$-th in its decision process for $p_i^+$. In the following,$[r]_+=\max\{r,0\}$ is the positive part, and the parameters $\alpha,\beta,\gamma,\delta>0$ are fixed. We also introduce the notation $p_i^-$ for the price that agent $i$-th is making the current day.

{\it i) Competition cost.} This cost takes into account the fact that choosing a value for $p_i^+$ lower than the (estimated) average should translate into a disadvantage:

\begin{equation}
\label{eq:competition}
c_1(\overline p,p_i^+)=\alpha[\tilde p_i-p_i^+]_+.
\end{equation}

{\it ii) Fidelity cost.} This cost takes account of the fact that choosing a value for $p_i^+$ greater than the (estimated) average decreases its attractiveness: 

\begin{equation}
\label{eq:fidelity}
c_2(\overline p,p_i^+)=\beta[p_i^+-\tilde p_i]_+.
\end{equation}

{\it iii) Reputation cost.} This cost takes account of the fact that changing the price often and significantly is seen as a bad behaviour by the customers:

\begin{equation}
\label{eq:reputation}
c_3(p_i^-,p_i^+)=\frac{\gamma}{2}(p_i^+-p_i^-)^2.
\end{equation}

{\it iv) Absolute price cost.} It penalizes excessively high values of the price in absolute terms, independently from the estimated average price:

\begin{equation}
\label{eq:absoluteprice}
c_4(p_i^+)=\frac{\delta}{2}(p_i^+)^2.
\end{equation}

The global cost to be minimized is then given by

\begin{equation}
\label{eq:cost}
c(p_i^-,\overline p,p_i^+)=c_1(\overline p,p_i^+)+c_2(\overline p,p_i^+)+c_3(p_i^-,p_i^+)+c_4(p_i^+)
\end{equation}

\noindent
and the minimization process of the $i$-th agent is given by

\begin{equation}
\label{eq:minimization}
p_i^+={\rm argmin}_{p\in\mathbb{R}}\mathbb{E}[c(p_i^-,\overline p,p)].
\end{equation}

\noindent
where $\mathbb{E}$ is the expectation value of the random variable $c$ (where the randomness is in the choice of the estimated average $\tilde p_i$ as in (\ref{eq:tildap})).

\begin{remark}
\label{rmrk:minimum}
The cost terms (\ref{eq:reputation}), (\ref{eq:absoluteprice}), in particular the latter, will play an essential role because they ensure the existence of an equilibrium, see Subsection \ref{subsec:equilibrium}. Indeed, as shown in Remark \ref{rem:gammadeltazero}, when $\gamma=\delta=0$ an equilibrium exists only if $\alpha=\beta$. Note that the parameter $\delta$ may somehow encode the information about the general price of petroleum: the bigger the second the smaller the first (it is less disadvantageous for stations to apply higher prices
). Concerning the minimization process (\ref{eq:minimization}) we see that it is posed on the whole real line $p\in\mathbb{R}$. We refer to Remark \ref{rmrk:calculation} for the positivity of the prices.

\end{remark}

\subsection{Well-posedness}
\label{subs:wellposedness}

 Let the agent $i$, $p_i^-$ and $\overline p$ be fixed. The expectation of  (\ref{eq:cost}) is then a function of $p\in\mathbb{R}$, denoted by $\varphi_{\gamma,\delta}(p)$ (dropping the index $i$). In the following $N(\cdot)$ is the distribution function of $\varepsilon^i$ and $N'(\cdot)$ its derivative, that is its density.

\begin{proposition}
\label{prop:phigammadelta}
The function $\varphi_{\gamma,\delta}(p)=\mathbb{E}[c(p_i^-,\overline p,p)]$ is derivable and

\begin{equation}\label{eq:varphiprimo}
\varphi_{\gamma,\delta}'(p)=-\alpha\left(1-N\left(\frac{p-\overline p}{\sigma_i}\right)\right)+\beta N\left(\frac{p-\overline p}{\sigma_i}\right)+\gamma(p-p_i^-)+\delta p
\end{equation}

\noindent
Moreover, it is strictly convex with second derivative

\begin{equation}
\label{eq:varphisecondo}
\varphi_{\gamma,\delta}''(p)=\frac{\alpha}{\sigma_i}N'\left(\frac{p-\overline p}{\sigma_i}\right)+\frac{\beta}{\sigma_i}N'\left(\frac{p-\overline p}{\sigma_i}\right)+\gamma+\delta
\end{equation}
\end{proposition}

\begin{proof}
We have

\[
\begin{array}{ll}
\displaystyle
\varphi_{\gamma,\delta}(p)=\mathbb{E}\left[\alpha[\tilde p_i-p]_++\beta[p-\tilde p_i]_++\frac{\gamma}{2}(p-p_i^-)^2+\frac{\delta}{2}p^2\right]=\\
\displaystyle
\alpha\int_{\overline p+\sigma_ix-p>0}(\overline p+\sigma_ix-p)N'(x)dx+\beta\int_{p-\overline p-\sigma_ix>0}(p-\overline p-\sigma_ix)N'(x)dx+\\
\displaystyle
\frac{\gamma}{2}(p-p_i^-)^2+\frac{\delta}{2}p^2=\alpha\int_{\frac{p-\overline p}{\sigma_i}}^{+\infty}(\overline p+\sigma_ix-p)N'(x)dx+\\
\displaystyle
\beta\int_{-\infty}^{\frac{p-\overline p}{\sigma_i}}(p-\overline p-\sigma_ix)N'(x)dx+\frac{\gamma}{2}(p-p_i^-)^2+\frac{\delta}{2}p^2.
\end{array}
\]

\noindent
The integrals in the right-hand side are $p$-dependent. Note that when $x=(p-\overline p)/\sigma_i$ the integrand function vanishes in both integrals. Deriving with respect to $p$ we then get (the term with the derivative of the boundary of the interval of integration vanishes, for what said above)

\begin{equation}
\begin{array}{ll}
\displaystyle
\varphi_{\gamma,\delta}'(p)=-\alpha\int_{\frac{p-\overline p}{\sigma_i}}^{+\infty}N'(x)dx+\beta\int_{-\infty}^{\frac{p-\overline p}{\sigma_i}}N'(x)dx+\gamma(p-p_i^-)+\delta p=\\
\displaystyle
-\alpha\left(1-N\left(\frac{p-\overline p}{\sigma_i}\right)\right)+\beta N\left(\frac{p-\overline p}{\sigma_i}\right)+\gamma(p-p_i^-)+\delta p.
\end{array}
\end{equation}

\noindent
From formula (\ref{eq:varphiprimo}) we see that $\varphi_{\gamma,\delta}$ is also twice derivable and we get (\ref{eq:varphisecondo}). Since $\alpha,\beta,\sigma_i,\gamma,\delta>0$ and $N'>0$, we eventually get the strict convexity.

\end{proof}

\begin{corollary}
\label{cor:varphi}
The function $\varphi_{\gamma,\delta}$ has a unique minimum point $p_i^+$ (the best-reply), identified as the unique point such that $\varphi_{\gamma,\delta}'(p_i^+)=0$.
\end{corollary}

\begin{proof} Use Proposition \ref{prop:phigammadelta} and $\lim_{p\to\pm\infty}\varphi_{\gamma,\delta}(p)=+\infty$.
\end{proof}

By Corollary \ref{cor:varphi}, the minimization process (\ref{eq:minimization}) is well posed for each single agent $i$ and for all $\alpha,\beta,\gamma,\delta>0,\sigma_i>0$ and random variable $\varepsilon^i$.

\begin{figure}[htbp]

    \centering
    \vspace{0.3cm}
    \begin{subfigure}{0.45\textwidth}
        \centering
        \includegraphics[width=\linewidth]{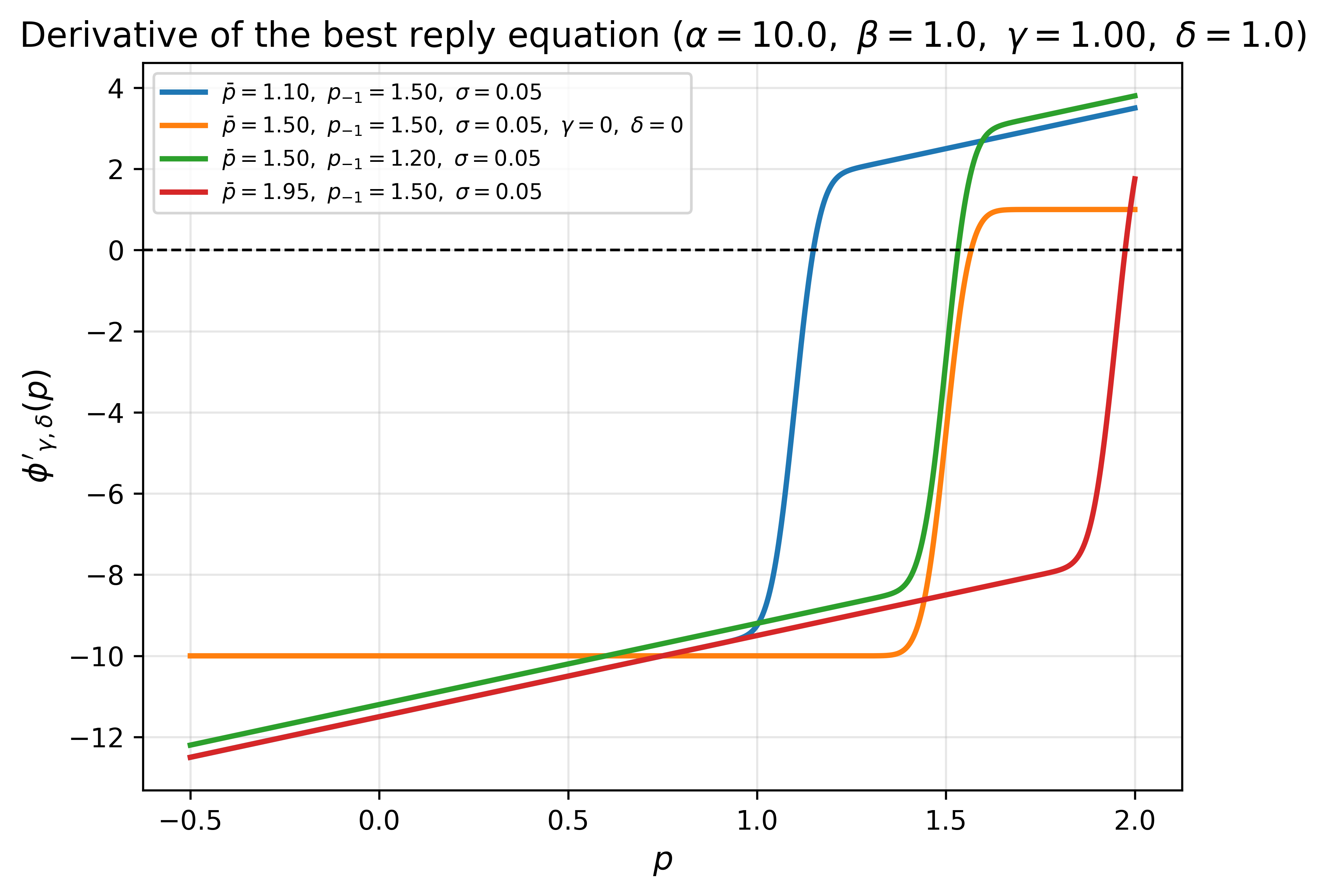}
        \caption{
        }
    \end{subfigure}
    \hfill
    \begin{subfigure}{0.45\textwidth}
        \centering
        \includegraphics[width=\linewidth]{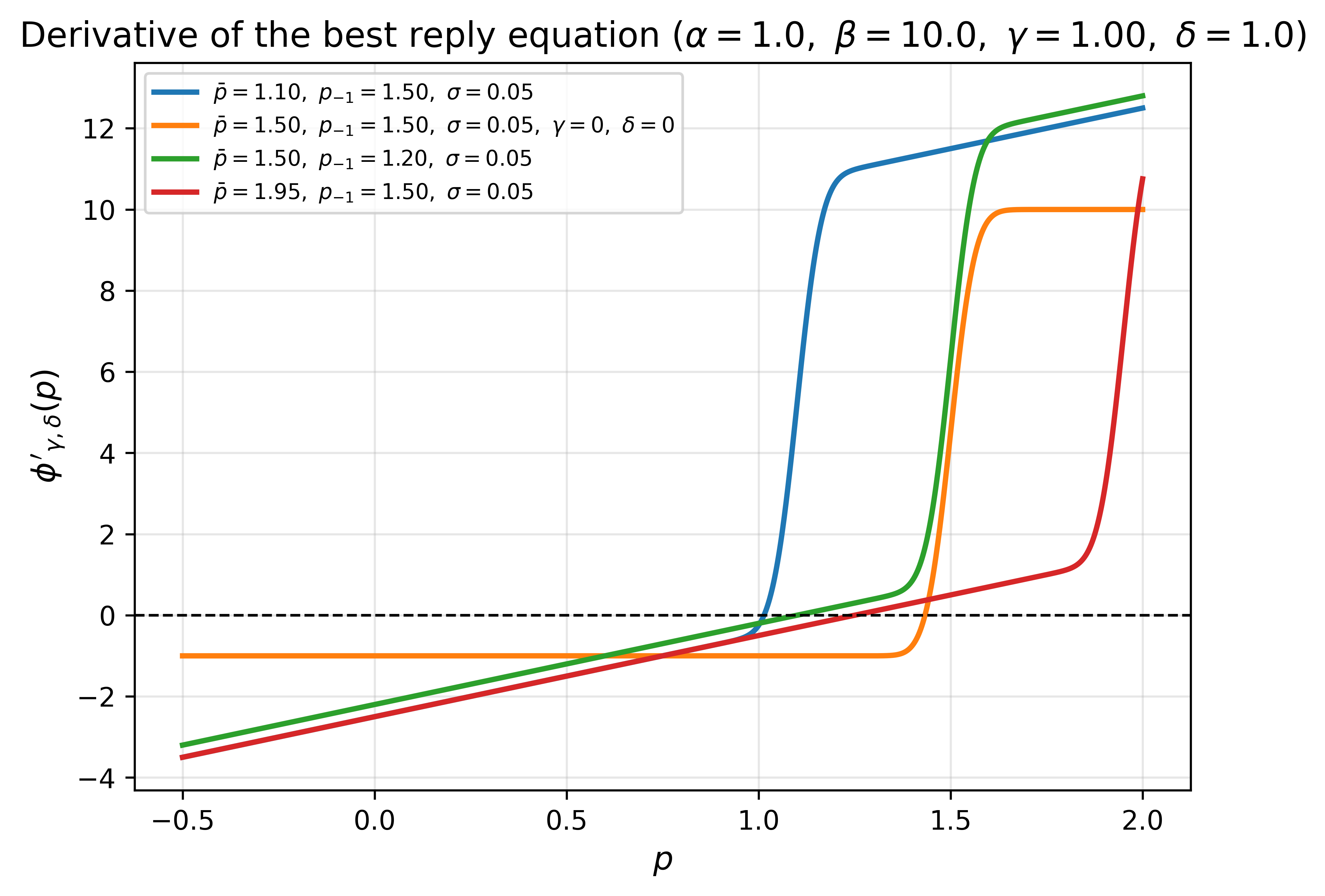}
        \caption{
        }
        \label{fig:best_reply_derivative2}
    \end{subfigure}
    \vspace{0.3cm}
    \begin{subfigure}{0.45\textwidth}
        \centering
        \includegraphics[width=\linewidth]{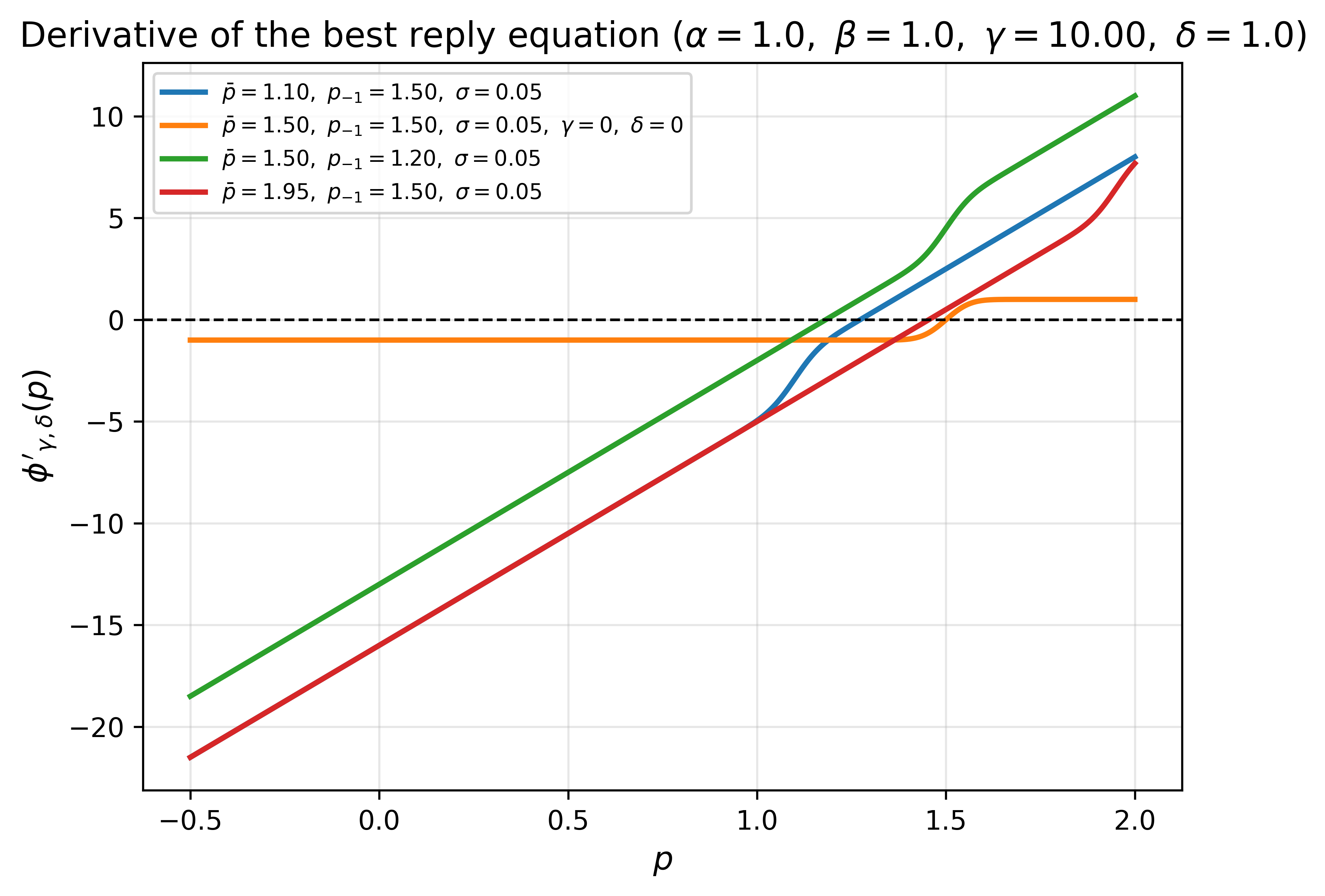}
        \caption{
        }
        \label{fig:best_reply_derivative2}
    \end{subfigure}
    \hfill
    \begin{subfigure}{0.45\textwidth}
        \centering
        \includegraphics[width=\linewidth]{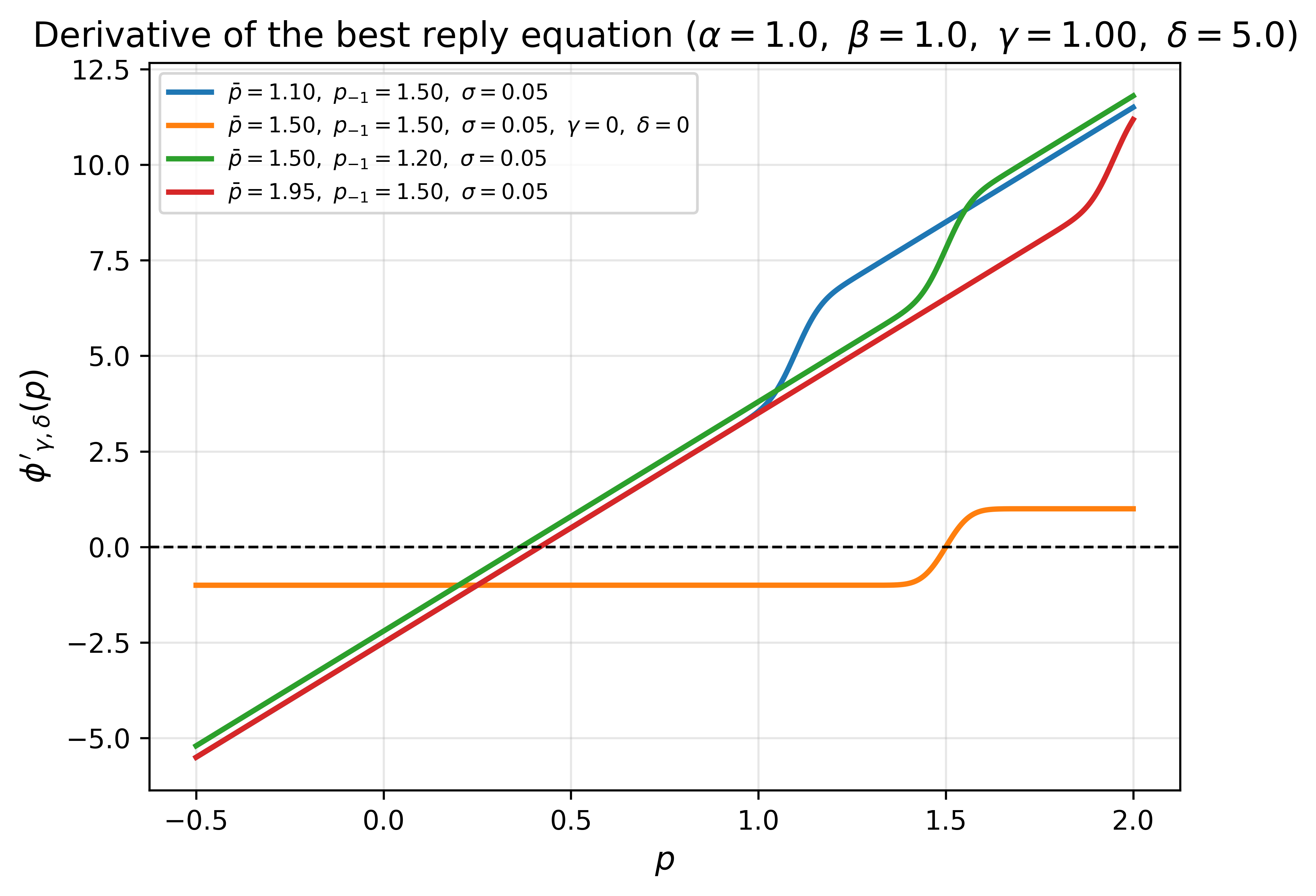}
        \caption{
        }
        \label{fig:best_reply_derivative2}
    \end{subfigure}
    
    \caption{a) If $\alpha$ is higher, the best-reply is higher; b) if $\beta$ is lower, the best-reply is lower; c) if $\gamma$ increases, the best-reply tends to not change; d) if $\delta$ is higher, the best-reply is lower.}
    \label{fig:best_reply}
\end{figure}

\subsection{Mean field equilibrium}
\label{subsec:equilibrium}

We label the discrete time evolution by $t_n=n$, $n\in\mathbb{N}$, with the meaning of the passing of the days. At the first day we are given the average price $\overline p_0$ and the collection of actual prices, one per every agent $i$, $\{p_{i,0}^-\}_i$, with $\overline p_0=(1/m)\sum_{i=1}^mp_{i,0}^-$, where $m\in\mathbb{N}\setminus\{0,1\}$ is the number of agents. Then, better specializing the transition (\ref{eq:transition}), we have the transition

\begin{equation}
\label{eq:transition2}
\mathbb{R}^m \ni(p_{i,n}^-)_i\longrightarrow\left((p_{i,n}^-)_i,\overline p_n\right)\longrightarrow (p_{i,n}^+)_i=(p_{i,n+1}^-)_i\in\mathbb{R}^m
\end{equation}

\noindent
where $\overline p_n=(1/m)\sum_{i=1}^mp_{i,n}^-$ and the values $p_{i,n}^+$ are constructed by the minimization process (\ref{eq:minimization}). We denote by $G:\mathbb{R}^m\to\mathbb{R}^m$ the function described in (\ref{eq:transition2}) which is well defined by the results in Subsection \ref{subs:wellposedness}.

\begin{definition}
A mean-field equilibrium $(p_i)_i\in\mathbb{R}^m$ is a fixed point of the map $G$, $G((p_i)_i)=(p_i)_i$.
\end{definition}

\begin{remark}\label{rem:gammadeltazero}
The terms (\ref{eq:reputation}) and (\ref{eq:absoluteprice}) are introduced not only for analytical reasons, such as coercivity of the minimization problem, but also to avoid a structural degeneracy of the equilibrium problem. Indeed, if $\gamma=\delta=0$, then \eqref{eq:varphiprimo} gives at equilibrium
$N((p_i-\overline p)/\sigma_i)=\alpha/(\alpha+\beta)$, hence
$p_i=\overline p+\sigma_i q$, with $q=N^{-1}(\alpha/(\alpha+\beta))$. Averaging with respect to $i$ yields
$q\,m^{-1}\sum_i\sigma_i=0$, and since $\sigma_i>0$ for every $i$, necessarily $q=0$.
Thus $\alpha/(\alpha+\beta)=N(0)=1/2$, namely $\alpha=\beta$. Hence, if
$\alpha\neq\beta$, the unregularized problem cannot admit an equilibrium; if
$\alpha=\beta$, the equilibrium is simply the actual value $\overline p$.
\end{remark}

\begin{proposition}
\label{prop:contraction}
The function $G$ is  a contraction map on $\mathbb{R}^m$, endowed with the infinity norm $\|\cdot\|_\infty$, and hence it has a unique fixed point. The game has then a unique mean-field equilibrium.
\end{proposition}

\begin{proof}
For every $i'=1,...,m$, $p_{i',n}^+=G_{i'}((p_{j,n}^-)_j)$ is the unique real number that makes $\varphi_{\gamma,\delta}'$ in (\ref{eq:varphiprimo}) vanish, with $\overline p=\overline p_n$ and $p_{i'}^-=p_{i',n}^-$. Differentiating the equality $\varphi_{\gamma,\delta}'(p_{i',n}^+((p_{i,n}^-)_i))=0$ we get for $i\neq i'$ and $i=i'$ respectively

\begin{equation}
\label{eq:jnoi}
\begin{array}{l}
\displaystyle
\frac{\partial p_{i',n}^+}{\partial p_{i,n}^-}=
\frac{\frac{1}{m}
\frac{\alpha+\beta}{\sigma_{i'}}N'\left(\frac{p_{i',n}^+-\overline p_n}{\sigma_{i'}}\right)}
{\frac{\alpha+\beta}{\sigma_{i'}}N'\left(\frac{p_{i',n}^+-\overline p_n}{\sigma_{i'}}\right)+\gamma+\delta}
\\
\displaystyle
\frac{\partial p_{i',n}^+}{\partial p_{i',n}^-}=
\frac{\gamma+\frac{1}{m}
\frac{\alpha+\beta}{\sigma_{i'}}N'\left(\frac{p_{i',n}^+-\overline p_n}{\sigma_{i'}}\right)}
{\frac{\alpha+\beta}{\sigma_{i'}}N'\left(\frac{p_{i',n}^+-\overline p_n}{\sigma_{i'}}\right)+\gamma+\delta}
\end{array}
\end{equation}

\noindent
Since all the entries in (\ref{eq:jnoi}) are non-negative, for every row $i'$ of the Jacobian matrix of $G$, namely $JG$, we have
\[
\sum_{i=1}^m\left|\frac{\partial p_{i',n}^+}{\partial p_{i,n}^-}\right|
=\frac{\gamma+\frac1m A_{i'}}{A_{i'}+\gamma+\delta}
+(m-1)\frac{\frac1m A_{i'}}{A_{i'}+\gamma+\delta}
=\frac{\gamma+A_{i'}}{A_{i'}+\gamma+\delta}<1,
\]
where, since $N'$ is bounded, $\displaystyle A_{i'}:=\frac{\alpha+\beta}{\sigma_{i'}}N'\left(\frac{p_{i',n}^+-\overline p_n}{\sigma_{i'}}\right)\leq C$ uniformly in $i'$, $\sigma_{i'}>0$ and  $p_{i',n}^+,\overline p_n\in\mathbb{R}$. We then get $L<1$ such that $\displaystyle \sup_{\xi\in\mathbb{R}^m}\vert\!\vert\!\vert JG(\xi)\vert\!\vert\!\vert_\infty\le L<1$. It then follows that $G$ is a contraction on $(\mathbb{R}^m,\|\cdot\|_\infty)$, which is a complete metric space, and then it has a unique fixed point (see e.g. \cite{ortrhe}).
\end{proof}

\begin{remark}
\label{rmrk:calculation}
By Proposition \ref{prop:contraction}, the map $G$ is contractive and therefore the mean-field equilibrium is unique; in fact it can be explicitly characterized, together with its uniqueness. At a fixed point $p_i^+=p_i^-=p_i$, the reputation term in \eqref{eq:varphiprimo} vanishes and one has
$-\alpha(1-N((p_i-\overline p)/\sigma_i))+\beta N((p_i-\overline p)/\sigma_i)+\delta p_i=0$.
Arguing as in Remark \ref{rem:gammadeltazero}, set $q_i=(p_i-\overline p)/\sigma_i$, so that $p_i=\overline p+\sigma_iq_i$ and $\sum_i\sigma_iq_i=0$.
The previous condition becomes 

\begin{equation}
    \label{eq:lambda}
(\alpha+\beta)N(q_i)+\delta\sigma_iq_i=\alpha-\delta\overline p=:\lambda.
\end{equation}

\noindent
For each $i$ the left-hand side is strictly increasing in $q_i$ and goes to $\pm\infty$ as $q_i\to\pm\infty$. Therefore for each $\lambda$ there exists a unique solution $Q_i(\lambda)$ with $\lambda\mapsto Q_i(\lambda)$ strictly increasing too. The equation $\sum_i\sigma_iQ_i(\lambda)=0$ then admits at most one solution $\lambda$.
Since $\lambda=(\alpha+\beta)/2$ gives $Q_i(\lambda)=0$ for every $i$, then it is the solution and $p_i=\overline p$ for all $i$. Substituting in (\ref{eq:lambda}),  $-\alpha/2+\beta/2+\delta\overline p=0$, hence the unique equilibrium is
$p_i=(\alpha-\beta)/(2\delta)$ for every $i$.
In particular, this equilibrium is positive, and therefore physically meaningful, if and only if $\alpha>\beta$. The case $\alpha<\beta$, in order to mantain a physical meaning, should require a different approach, namely a constraint minimization problem in Subsection \ref{subs:wellposedness} with possible consequences in Subsection \ref{subsec:equilibrium}.\\ 
Moreover, the equilibrium value does not depend on $\gamma$, since the reputation term vanishes at a fixed point, nor on the coefficients $\sigma_i$. Obviously, they anyway affect the off-equilibrium best replies and hence the rate of convergence. Also, Proposition \ref{prop:contraction} gives stability of such equilibrium and guarantees that it can be reached by successive iterations. In Figure 1 the best-reply functions for a fixed agent $i$ are reported for different values of the parameters. Note that the case $\gamma=\delta=0$ is also there reported. Indeed, the results of Subsection \ref{subs:wellposedness}, in particular Corollary \ref{cor:varphi}, also hold in that case and moreover one easily gets the convergence of minima as $(\gamma,\delta)\to(0,0)$ for $(\overline p, p_i^-)$ fixed.

\end{remark}

\section{Simulations}
\label{sec:simulations}

To account for heterogeneity in the model, we introduce multiple station clusters. Let $K$ denote the number of clusters. We then fit a model with $K$ quadruples of parameters $(\alpha_\nu,\beta_\nu,\gamma_\nu,\delta_\nu)$. Let $m_\nu$ denote the size of cluster $\nu$. The average price is still calculated on the totality of agents $m=m_1+\cdots+m_K$. Existence and uniqueness of the equilibrium follow as in  Proposition \ref{prop:contraction}.

Over the years, the Agenzia delle Entrate has considered a station classification procedure, \cite{Age}, which grouped petrol stations according to their available services (e.g.\ bars, mechanical assistance, car washing facilities, etc.), with the aim of predicting their annual revenue, suggesting up to 7 groups and also giving their numerosity. This does not necessarily translate into 7 clusters for the price, but can be used as an upper bound on $K$.


To avoid the risk of overfitting and to 
limit model complexity, only three clusters were used, resulting in a total of 12 parameters. Stations were assigned to clusters according to their observed prices on the first day of the quarter as shown in \ref{fig:distribution} (in the Province of Trento, $m=173$ gas stations). The model assumes that cluster membership remains constant over time. In this case the K-means algorithm was used, but further directions of the project include a more rigorous justification of the number of clusters, 
and identifiability of the parameters. The reported values should therefore be regarded as one representative calibration within the admissible parameter space, rather than as uniquely identified estimates. The code used to manipulate the data, simulate the trajectories, and estimate the parameters is available at \cite{fedP}. We also note that the
 \emph{best replies} are computed  by \eqref{eq:varphiprimo} as in Corollary \ref{cor:varphi}; hence, the simulations are deterministic. 

{\it \underline{Parameters estimation}}.
 Let $\boldsymbol{p_0}=(\bar{p}_{0},(p_{0,i}^-)_i)\in \mathbb{R}^{1+173}$ be the initial condition. For every choice of parameters $\boldsymbol{\theta}=((\alpha_{\nu})_\nu, (\beta_\nu)_\nu,(\gamma_\nu)_\nu,(\delta_\nu)_\nu)$, we have a simulated trajectory $\left(\bar{P}_{n}(\boldsymbol{\theta};\boldsymbol{p_0}) \right)_n$. As objective function for the estimation we take the square loss (\ref{eq:loss}), possibly computed over a subset $\mathcal{D'}$ of the total dataset $\mathcal{D} =\{\bar{p}_n: n=1,..., T\}$, whereas, once specified a parameter space $\Theta$, the optimal value $\boldsymbol{\theta^*}$ is as in (\ref{eq:theta*}): 
 \begin{equation}
 \label{eq:loss}
     \mathcal{L}(\boldsymbol{\theta};\mathcal{D}')=\frac{1}{2}\sum_{\bar{p}_n \in \mathcal{D'}}(\bar{P}_n(\boldsymbol{\theta}; \boldsymbol{p_0})-\bar{p}_n)^2
 \end{equation}

\begin{equation}
\label{eq:theta*}
    \boldsymbol{\theta^*}=\operatorname*{arg\,min}_{\boldsymbol{\theta}\in\Theta} \mathcal{L}(\boldsymbol{\theta};\mathcal{D}')
\end{equation}
 The parameters $\sigma_i$ are estimated directly from the data as the standard deviation of the observations for each station, that is $\hat{\sigma}^2_i=\frac{1}{92-1} \sum_{n=1}^{92}(p_{n,i}-{p}_i^*)^2$, where $p_{n,i}$ is the price made by station $i$ at day $n$, and ${p}_i^*$ is the average price for station $i$ across the 92 days.  The parameter space was taken to be $\Theta=[10^{-6},100]^{12}$, and $\mathcal{D}'=\{\bar{p}_n:n=32,33,...,k,k+1,...,64,65\}$. The estimation problem should not be interpreted as fitting a 12-parameter curve to 33 observations, as each observation of the simulated average price is generated through the interaction of all 173 agents solving their individual optimization problems. 
 The dataset was filtered to exclude stations with missing values. Since the average price was not recomputed after filtering, the simulated and observed trajectories differ slightly at the beginning of the time horizon. Table \ref{tab:parameters} reports the estimated parameters, while Figure \ref{fig:model-fit} compares the simulated and observed average prices. Parameter estimation was performed using the Trust Region Reflective algorithm implemented in \texttt{scipy.optimize.least\_squares}. 
 The optimization terminated with a first-order optimality: $3.16 \times 10^{-6}$. The number of significant digits reported in Table~\ref{tab:parameters} should not be interpreted as parameter uncertainty estimates.
 
  Figure \ref{fig:model-fit} shows that the real behavior of the prices differs from the qualitative one predicted by our model (and probably from the one desired by the government) after approximately one month. Other exogenous factors are indeed stronger than the forcing costs in our analytical model (\ref{eq:cost}), as can be inferred by Table~\ref{tab:parameters} where the reputation coefficients $\gamma_i$, equal to the lower bound $10^{-6}$, show their limited contribution in that real specific situation.
Also note that Cluster 0 and 1 show relatively similar $\beta$'s, suggesting a similar concern about fidelity. Additionally, Cluster 2 has the lowest $\alpha$ and the highest $\beta, \delta$, suggesting that for \emph{premium}-type petrol stations, fidelity and absolute price cost play a major role compared to competition, which appears to be more relevant for Cluster 1.

 \begin{figure}[htbp]
\centering

\begin{subfigure}[t]{0.46\textwidth}
    \centering
    \includegraphics[height=4.0cm]{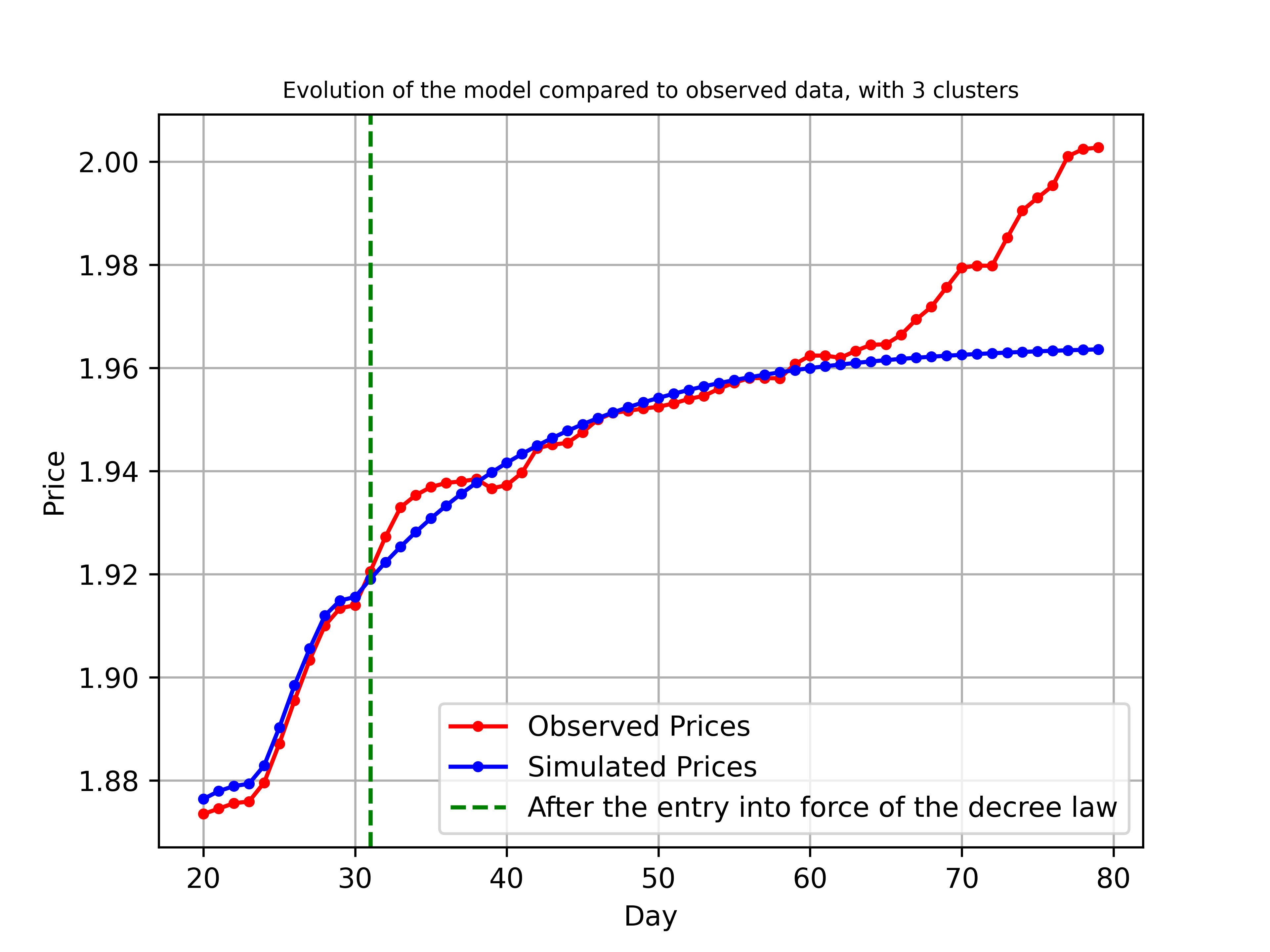}
    \caption{Model fit}
    \label{fig:model-fit}
\end{subfigure}
\hfill
\begin{subfigure}[t]{0.46\textwidth}
    \centering
    \includegraphics[height=4.0cm]{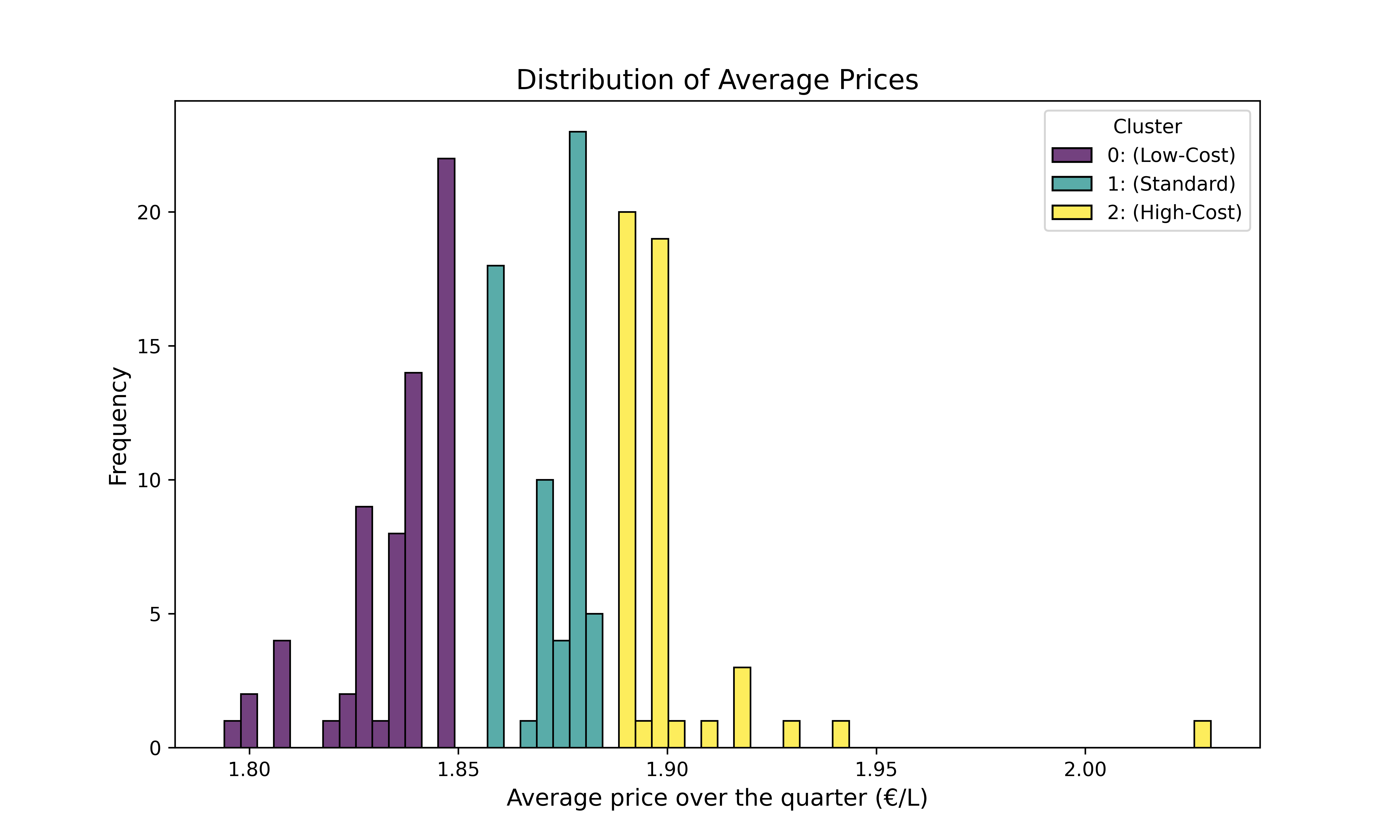}
    \caption{Distribution of prices on the $1^{\mathrm{st}}$ July 2023.}
    \label{fig:distribution}
\end{subfigure}

\caption{}

\end{figure}

\begin{table}[!ht]
\centering
\renewcommand{\arraystretch}{1.3}

\begin{tabular}{cccccc}
\toprule
Cluster & N$^o$ agents & $\alpha$ & $\beta$ & $\gamma$ & $\delta$\\
\midrule
0 & 64 & 54.3 & 0.291 & 0 & 13.2\\
1 & 61 & 67.2 & 0.287 & 0 & 7.57\\
2 & 48 & 39.1 & 10.5 & 0  & 20.4\\
\bottomrule
\end{tabular}

\caption{Final cost: $1.4893\cdot 10^{-4} \text{(euro/l)}^2$ ; RMSE: $3.004\cdot 10^{-3}$ euro/l; first-order optimality measure of $3.16 \times 10^{-6}$}
\label{tab:parameters}

\end{table}





\end{document}